\documentclass[12pt, a4]{amsart}

\usepackage{latexsym}
\usepackage{amsmath}
\usepackage{amssymb}
\usepackage{amsbsy}
\usepackage[dvips]{graphicx,color}
\usepackage[all]{xy}
\def\R {\mathbb{R}}
\def\Z {\mathbb{Z}}

\def\C {\mathbb{C}}

\title{On a homology of foliations defined by non-singular Morse-Smale flows}

\author{Masato Akizawa, Ryosuke Furuta and Shigeaki Miyoshi}

\address{c/o Department of Mathematics\\ Chuo University\\
   1-13-27 Kasuga Bunkyo-ku, Tokyo\\ 112-8551, Japan}
\address{c/o Department of Mathematics\\ Chuo University\\
   1-13-27 Kasuga Bunkyo-ku, Tokyo\\ 112-8551, Japan}
\address{Department of Mathematics\\ Chuo University\\
   1-13-27 Kasuga Bunkyo-ku, Tokyo\\ 112-8551, Japan}
\email{miyoshi@math.chuo-u.ac.jp}

\keywords{non-singular Morse-Smale flow, Morse homology}
\subjclass[2020]{primary~57R30, secondary~57R99, 57M99, 57R25}
\newtheorem{thm}{Theorem}[section]    % Standard theorem environment
\newtheorem{lem}[thm]{Lemma}          % Lemma environment with numbering 
\newtheorem{prop}[thm]{Proposition}
\newtheorem{Thm}{Theorem}%[section] % 1st argument is your name for it
\newtheorem{Assmptn}[thm]{Assumption}      % Unnumbered environment for assumption
\theoremstyle{definition}
\newtheorem{defn}[thm]{Definition}     % Definition environment with 
\newtheorem{claim}[thm]{Claim}

\begin{document}

\baselineskip=18pt

\begin{abstract}    % abstract
We propose a definition of a homology of a one-dimensional oriented foliation defined by a non-singular Morse-Smale flow.  We also show the calculation of the homology of such a foliation which is naturally associated with a Seifert fibration.  
\end{abstract}

\maketitle

%%%%%%%%%%%%%%%%%%%%   Start of main body of article
\section{Introduction} 
\label{intro}
The purpose of this paper is to propose a definition of a homology of non-singular Morse-Smale flows.  Indeed, since it does not depend on parameters of the flow, it would be an invariant of the one-dimensional oriented foliations each of which consists of flow-lines of a non-singular Morse-Smale flows as leaves.  We will call such a foliation {\em non-singular Morse-Smale foliation}.  

Suppose a non-singular Morse-Smale flow is given.  Then we have a round handle decomposition of the underlying manifold which is defined in a similar way as a gradient vector field of a Morse function defines an ordinary handle decomposition.  Under an assumption on the order of indices of periodic orbits (which implies that round handles are attached in line), it induces a filtration of the manifold by indices of the periodic orbits.  Then we could have something like a homology from this filtration.  In fact, we have a chain complex determined by the filtration and its homology is defined.  Thus we have a homology of round handle decompositions.  However, it may be obscure how the flow (or the foliation defined by the flow) determines directly the homology thus obtained.  Therefore we consider the Conley index of the periodic orbits and connecting annuli between them and describe those relations in a homological sense.  The description leads us to the definition of the boundary operator which reflects the structure of flow-lines (i.e., the foliated structure) and then we reach a homology of non-singular Morse-Smale foliations.  

The paper is organized as follows.  Section \ref{NMS_RHD} is a preliminary on the results of non-singular Morse-Smale flows and round handle decompositions due to D. Asimov \cite{Asimov_RHD}, \cite{Asimov_StrStable}, and J. Morgan \cite{Morgan}.  In Section \ref{Homology_RHD}, we introduce a definition of the homology of round handle decompositions under an assumption on the arrangement of attaching round handles as mentioned above.  In Section \ref{ConleyIndex_Bdry}, we investigate the aspect related to the boundary operators of our chain complex of round handle decompositions in view of Conley indices, which leads us to the definition of the desired boundary operator. For an example of a calculation of the homology, we exhibit a case of non-singular Morse-Smale flows naturally associated with Seifert fibrations in Section \ref{SF}.   

In this paper, we work in the smooth category and all the manifolds are orientable, unless otherwise stated. The coefficient of homology groups is $\Z$.  

\section{Non-singular Morse-Smale flows and round handle decompositions}\label{NMS_RHD}
Let $M$ be a compact manifold of dimension $n$, possibly with boundary.  A 
flow $\varphi$ on $M$ is called a {\em Morse-Smale flow} if it is generated by a vector field $X$ which satisfies the following conditions:
\begin{enumerate}
\item 
$X$ has a finite number of critical elements (singular points and periodic orbits), all of which are hyperbolic,
\item
for any critical elements $\sigma_{1}, \sigma_{2}$ of $X$, the unstable submanifold $W^{u}(\sigma_{1})$ is transversal to the stable submanifold $W^{s}(\sigma_{2})$,
\item
the nonwandering set $\Omega (X)$ of $X$ is equal to the union of the critical elements of $X$.  
\end{enumerate}
For the detailed property of Morse-Smale vector fields and flows, we refer the readers to \cite{Palis-Melo}.  
We concern a one-dimensional oriented foliation which consists of all the orbits (flow-lines) of non-singular Morse-Smale flows (NMS flows for short) as leaves.  We call such a foliation an {\em NMS-foliation}.  

Suppose $M$ is a compact connected manifold with boundary (possibly empty) $\partial M$, and a union $\partial_{-}M$ of connected components of the boundary is specified.  We call such a pair a {\em manifold pair} in this paper.  
%Then, $(M, \partial_{-}M)$ is called a {\em flow manifold} if there exists a non-singular vector field on $M$ transverse to $\partial M$ pointing inward on $\partial_{-}M$ and outward on $\partial_{+}M = \partial M\setminus\partial_{-}M$.  The necessary and sufficient condition to be a flow manifoldis its relative Euler characteristic $\chi (M, \partial_{-}M)$ vanishes.  
A vector field and a flow on a manifold pair $(M, \partial_{-}M)$ are assumed to be transversal to $\partial M$ and pointing outward on $\partial_{-}M$ and inward on $\partial_{+}M = \partial M\setminus\partial_{-}M$, unless otherwise stated.  Note that the necessary and sufficient condition for the existence of a {\em non-singular} vector field on $(M, \partial_{-}M)$ is vanishing of its relative Euler characteristic $\chi (M, \partial_{-}M)$. 
 
Next, we define a round handle decomposition of $(M, \partial_{-}M)$.  Let $E^{k}$ denote a $k$-disk bundle over the circle $S^{1}$ where $k$ is an integer with $0\leq k\leq n-1 ; n=\mathrm{dim}M$.  Set $R^{k} = E^{k}\oplus E^{n-k-1}$.  Write $\partial_{-}R^{k} = \partial E^{k}\times_{S^{1}}E^{n-k-1}$ and $\partial_{+}R^{k} = E^{k}\times_{S^{1}}\partial E^{n-k-1}$.  We call $R^{k}$ or $(R^{k}, \partial_{-}R^{k})$ an {\em $n$-dimensional round handle of index $k$} or {\em $n$-dimensional round $k$-handle}. The zero-section of a round $k$-handle is called the {\em core circle} of the round handle.
Note that $R^{k}$ is orientable if and only if the both of the disk bundles $E^{k}$ and $E^{n-k-1}$ are trivial or non-trivial simultaneously.  We say the former is {\em trivial} and the latter is {\em twisted}.  

Let $(W, \partial_{-}W)$ be an $n$-dimensional manifold pair.  
For an embedding $\alpha :\partial_{-}R^{k}\rightarrow \partial_{+}W$, we define the identification space $W\cup_{\alpha}R^{k}$, the quotient space of the disjoint union $W\sqcup R^{k}$ identified by $\alpha$.  The resulting manifold is also denoted by $W +_{\alpha}R^{k}$ or simply by $W + R^{k}$.  Set $\partial_{-}(W +_{\alpha} R^{k}) = \partial_{-}W$ and note that $\partial_{+}(W +_{\alpha} R^{k}) = (\partial_{+}W\setminus\mathrm{Im}\alpha)\cup\partial_{+}R^{k}$.  We say that $W +_{\alpha}R^{k}$ is obtained from $W$ by {\em attaching a round $k$-handle $R^{k}$}.   
A {\em round handle decomposition} (RHD for short) of a manifold pair $(M, \partial_{-}M)$ is a representation 
\[
 M = (\partial_{-}M\times [0, 1])+R^{k_{1}}+\cdots +R^{k_{m}}
\]
which denotes successive operations of attaching round handles.  We say an RHD is {\em simple} if all the round handles appearing in the RHD are trivial.  
\vspace{2ex}

The following propositions are foundations of our research. 
%\vspace{2ex}

\begin{prop}[Asimov \cite{Asimov_RHD}]
Let $(M, \partial_{-}M)$ is a manifold pair with $\chi (M, \partial_{-}M) = 0$.  Then, $(M, \partial_{-}M)$ has an RHD if $\mathrm{dim}M \neq 3$.  Moreover, if $M$ is not a M\"{obius band} $($and $\mathrm{dim}\neq 3)$, then $(M, \partial_{-}M)$ has a simple RHD.
\end{prop}
%\vspace{2ex}

\begin{prop}[Morgan \cite{Morgan}]
 Suppose $(M, \partial_{-}M)$ is a manifold pair and $\mathrm{dim}M = 3$.  Then, $(M, \partial_{-}M)$ has an RHD if and only if $M$ is a graph manifold.  
\end{prop}
%\vspace{1ex}

\begin{prop}[Asimov \cite{Asimov_RHD}]
If a  manifold pair $(M, \partial_{-}M)$ has an RHD, then $(M, \partial_{-}M)$ has an NMS flow whose periodic orbits are exactly the core circles of the RHD.   
\end{prop}
%\vspace{1ex}

\begin{prop}[Morgan \cite{Morgan}]\label{NMS2RHD}
If a manifold pair $(M, \partial_{-}M)$ has an NMS flow, then $(M, \partial_{-}M)$ has an RHD whose core circles are the periodic orbits of the flow.
\end{prop}
\vspace{2ex}

Suppose that an NMS flow $\varphi$ is given on a manifold pair $(M, \partial_{-}M)$ and let $c_{1}$ and $c_{2}$ be periodic orbits of $\varphi$.  Suppose the index of $c_{i}$ is $k_{i}$ ($i = 1, 2$).  Since the unstable submanifold $W^{u}(c_{1})$ of $c_{1}$ and the stable submanifold $W^{s}(c_{2})$ of $c_{2}$ are transversal to each other, we have $W^{u}(c_{1})\cap W^{s}(c_{2}) = \emptyset$ if $k_{1} < k_{2}$.  Note that (un)stable submanifolds are {\em saturated}.  Recall that a subset of $M$ is {\em saturated} if it is a union of flow-lines, i.e., orbits of $\varphi_{t}$.  Since we are concerned with a foliation consisting of flow-lines of a flow, we mainly say {\em saturated} instead of {\em invariant}.  From now on, for a technical reason, we assume the following.  Set $W(c_{1}, c_{2}) = W^{u}(c_{1})\cap W^{s}(c_{2})$.  
%
%\begin{center}
%{\it
% Assumption: $W(c_{1}, c_{2}) = \emptyset$ if indices of $c_{1}$ and $c_{2}$ are equal.  
%}
%\end{center}
%
\vspace{2ex}
\begin{Assmptn}
$W(c_{1}, c_{2}) = \emptyset$ if indices of $c_{1}$ and $c_{2}$ are equal.  
\end{Assmptn}
\vspace{2ex}
\noindent
With this assumption, for any NMS flow $\varphi$ the RHD associated with $\varphi$ can be 'totally ordered'.  More precisely, the associated RHD can be represented index-ascendingly as $ M = (\partial_{-}M\times [0, 1])+\sum_{i=1}^{m_{0}}R^{0}_{i}+\cdots +\sum_{i=1}^{m_{n-1}}R^{n-1}_{i}$ and the round handles with the same indices attached independently, that is, $\partial_{-}R^{k}_{i}\cap\partial_{+}R^{k}_{j} = \emptyset$.  
Thus, we have a filtration $\emptyset = M^{(-1)}\subset M^{(0)}\subset M^{(1)}\subset\cdots\subset M^{(n-1)} = M$, where $M^{(0)} = (\partial_{-}M\times [0, 1]) + \sum_{i=1}^{m_{0}}R^{0}_{i}$ and $M^{(k)} = M^{(k-1)} + \sum_{i=1}^{m_{k}}R^{k}_{i}$.  

\section{A homology of round handle decompositions}\label{Homology_RHD}
From now on, we assume that $M$ is a closed manifold throughout the paper.  
In this section, all round handles are assumed to be trivial.  
Let $M = \sum_{i=1}^{m_{0}}R^{0}_{i}+\cdots +\sum_{i=1}^{m_{n-1}}R^{n-1}_{i}$ be a simple RHD of $M$ and $\rho (M)$ denote this RHD.  Suppose $\emptyset = M^{(-1)}\subset M^{(0)}\subset M^{(1)}\subset\cdots\subset M^{(n-1)} = M$ is the associated filtration.  

\begin{prop}\label{RHD_ChainGrp}
For an integer $k$ with $0\leq k\leq n-1$, the relative homology group $H_{k+1}(M^{(k)}, M^{(k-1)})$ is a free Abelian group of rank $m_{k}$.
\end{prop}

\noindent
This proposition follows from the following lemma.  

\begin{lem}
Suppose that $(Z, \partial_{-}Z)$ is an $n$-dimensional manifold pair and $Y$ is obtained by attaching a trivial round $k$-handle to $\partial_{+}Z$: $Y = Z+R^{k}$.  Then the relative homology group $H_{k+1}(Y, Z)$ is isomorphic to $H_{k+1}(R^{k}, \partial_{-}R^{k})$, which is an infinite cyclic group.  
\end{lem}
\begin{proof}
Consider a collar of $\partial_{-}R^{k}$ in $R^{k}$ as $\partial_{-}R^{k}\times [0, \epsilon ]$, where $\partial_{-}R^{k}\approx\partial_{-}R^{k}\times\{ 0\}$ and $\epsilon$ is a small positive number.  Then a pair $(Y, Z\cup(\partial_{-}R^{k}\times [0, \epsilon ]))$ is homotopy equivalent to $(Y, Z)$. Moreover, $Z$ is closed and $\mathrm{Int}(Z\cup(\partial_{-}R^{k}\times [0, \epsilon])) = Z\cup(\partial_{-}R^{k}\times [0, \epsilon ) )$ in $Y$.  Hence by excision, $H_{k+1}(Y, Z)$ is isomorphic to $H_{k+1}(Y\setminus Z, \partial_{-}R^{k}\times (0, \epsilon])\cong H_{k+1}(R^{k}, \partial_{-}R^{k})\cong H_{k+1}(S^{1}\times D^{k}, \partial (S^{1}\times D^{k}))\cong\Z$.  
\end{proof}
\noindent
In view of Proposition \ref{RHD_ChainGrp}, we can consider that symbolically the homology group $H_{k+1}(M^{(k)}, M^{(k-1)})$ is generated by the core circles of the round $k$-handles.  Recall that for a round $k$-handle $(R^{k}, \partial_{-}R^{k}) = (S^{1}\times D^{k}\times D^{n-k-1}, S^{1}\times\partial D^{k}\times D^{n-k-1})$, its {\em core circle} is $S^{1}\times \{ 0\}\times \{ 0\}\subset R^{k}$ and $S^{1}\times D^{k}\times\{ 0\}$ is a real cycle representing a generator of $H_{k+1}(R^{k}, \partial_{-}R^{k})$.  

It may be appropriate that we fix our orientation convention here. Suppose that $M$ is oriented.  Let $c^{k}$ be a periodic orbit of an NMS flow and $(R^{k}, \partial_{-}R^{k})$ an associated round $k$-handle.  Suppose that the splitting $T_{x}c^{k}\oplus E^{\mathrm{uu}}_{x}\oplus E^{\mathrm{ss}}_{x}$ of $T_{x}M$ at every $x\in c^{k}$ is given.  Here, $E^{\mathrm{uu}}_{x}$ (resp. $E^{\mathrm{ss}}_{x}$) is the subspace of eigenvalues greater than $1$ (resp. less than $1$) of the differential of Poincar\'{e} map along $c^{k}$, thus $E^{\mathrm{uu}}$ (resp. $E^{\mathrm{ss}}$)  is a subbundle of $TM|c^{k}$ of rank $k$ (resp. of rank $n-k-1$) and $Tc^{k}\oplus E^{\mathrm{uu}}$ (resp. $Tc^{k}\oplus E^{\mathrm{ss}}$) is the tangent plane field of unstable manifold $W^{u}(c^{k})$ (resp. stable manifold $W^{s}(c^{k})$) along $c^{k}$.  We then choose and fix an orientation $[v_{0}, v_{1},\ldots ,v_{k}]$ of $T_{x}c^{k}\oplus E^{\mathrm{uu}}_{x}$ arbitrarily, where $v_{0}\in T_{x}c^{k}$ is the flow direction and $v_{1},\ldots , v_{k}\in E^{\mathrm{uu}}_{x}$ define any orientation of $E^{\mathrm{uu}}_{x}$.  The orientation of $E^{\mathrm{ss}}_{x}$ is determined automatically by the orientation of $M$.  We extend this orientation to the whole $c^{k}$.  Hence we have a canonical isomorphism $H_{k+1}(R^{k}, \partial_{-}R^{k})\cong\Z$ with this orientation.  

Now suppose that $\partial_{k+1} : H_{k+1}(M^{(k)}, M^{(k-1)})\rightarrow H_{k}(M^{(k-1)}, M^{(k-2)})$ is the connecting homomorphism in the long exact sequence of the triple $(M^{(k)}, M^{(k-1)}, M^{(k-2)})$.  
Set $C^{^\mathrm{RHD}}_{k}(\rho (M)) = H_{k+1}(M^{(k)}, M^{(k-1)})$ and $\partial^{^\mathrm{RHD}}_{k} = \partial_{k+1}$.  
Note that the indices are shifted.  The following is clear.  

\begin{prop}
The graded Abelian groups $\{ C^{^\mathrm{RHD}}_{k}(\rho (M)), \partial^{^\mathrm{RHD}}_{k}\}$ is a chain complex.  
\end{prop}

\begin{Thm}
For an RHD $\rho (M)$ of $M$, we have a graded group denoted by $H^{^\mathrm{RHD}}_{\ast}(\rho (M))$, which is defined to be the homology of the chain complex $\{ C^{^\mathrm{RHD}}_{k}(\rho (M)), \partial^{^\mathrm{RHD}}_{k}\}$.  
\end{Thm}

\begin{defn}
We call the graded group $H^{^\mathrm{RHD}}_{\ast}(\rho (M))$ a {\em homology} of the RHD $\rho (M)$.  
\end{defn}

Suppose an NMS flow is given, then we have the associated RHD by Proposition \ref{NMS2RHD}.  Hence we have the homology as above.  However, the strict relation between NMS flows and their associated RHDs is not clear.  More precisely, though isotopic RHDs induce isomorphic homology groups, the structures of flow-lines (i.e. oriented one-dimensional foliations) corresponding to the RHDs might change.  Therefore, 
for an NMS-foliation, 
we need a more precise 
%filtration which describes the foliation defined by an NMS flow.  
description of the boundary operator depending on the foliation structure itself.   

\section{Conley index and boundary operators}\label{ConleyIndex_Bdry}
In this section, we use the notion called {\em Conley index} in order to define another boundary operator.  Suppose a flow $\varphi_{t}$ is given on a closed manifold $M$.  Throughout this section, we write simply $x\cdot t$ (resp. $N\cdot t$) to be $\varphi_{t}(x)$ (resp. $\varphi_{t}(N)$) for $x\in M$ (resp. $N\subset M$) and $t\in\R$.    

\subsection{Index pair and Conley index}
\begin{defn}
For any subset $N\subset M$, we define the {\em maximal saturated set} $I(N)$ of $N$ for $\varphi_{t}$ to be 
\[
I(N) = \{ x\in N\ |\ x\cdot t\in N\ \textrm{for any}\ t\in\R\}. 
\]
Note that $I(N) = \bigcap_{t\in\R}N\cdot t$.  
\end{defn}

\begin{defn}
A compact saturated subset $S\subset M$ is said to be {\em isolated} if there exists a compact neighborhood $N$ of $S$ such that $I(N) = S$.  In the case, $N$ is called an {\em isolating neighborhood} for $S$.  
\end{defn}

\begin{defn}
An {\em index pair} $(N, L)$ for a isolated compact saturated set $S$ is a pair of compact subset $L\subset N$ such that the following conditions are satisfied:
\begin{enumerate}
\item 
$S=I(\mathrm{Cls}(N\setminus L))\subset\mathrm{Int}(N\setminus L)$, where $\mathrm{Cls}$ and $\mathrm{Int}$ denote the closure and the interior respectively.  
\item\label{pos_inv}
If $x\in L$ and $x\cdot [0, t]\subset N$, then $x\cdot [0, t]\subset L$.  
\item\label{exit_set}
If $x\in N\setminus L$, then there is a positive number $t$ such that $x\cdot [0, t]\subset N$.
\end{enumerate}
The condition (\ref{pos_inv}) implies $L$ is positively invariant by the flow and the condition (\ref{exit_set}) implies that each orbit that leaves $N$ must go through $L$ first.  $L$ is called the {\em exit set} of the index pair.  
\end{defn}

Note that a periodic orbit of an NMS flow is a isolated compact saturated set and the associated round handle $(R^{k}, \partial_{-}R^{k})$ is an index pair for the periodic orbit.  

The following results were proven by Conley in \cite{Conley} and by Salamon in \cite{Salamon} with new and simplified definitions and proofs.  

\begin{prop}\label{existence_indexpair}
Every isolated compact saturated set admits an index pair. 
\end{prop}

\begin{prop}\label{homotopy_inv_indexpair}
If $S$ is an isolated compact saturated set and $(N, L)$ and $(N', L')$ are two index pairs for $S$, then $N/L$ and $N'/L'$ are homotopy equivalent as pointed spaces via maps that are induced by the flow.  
\end{prop}

\begin{defn}
The {\em Conley index} of an isolated compact saturated set $S$ is the homotopy type of $N/L$ where $(N, L)$ is an index pair for $S$.  
\end{defn}

We say that an index pair $(N, L)$ is {\em regular} if the inclusion map $L\hookrightarrow N$ is a cofibration.  Any index pair $(N, L)$ can be modified into a regular index pair $(N, \tilde{L})$ (see Section 5.1 of \cite{Salamon}).  Suppose that $S$ is an isolated compact saturated set and $(N, L)$ is a regular index pair for $S$.  Since for any cofibration $B\hookrightarrow Y$ with a closed subset $B$, the quotient map $(Y, B)\rightarrow (Y/B, \ast)$ induces isomorphisms between homology (resp. cohomology) groups of $(Y, B)$ and $(Y/B, \ast)$,  $H_{\ast}(N, L)$ (resp. $H^{\ast}(N, L)$) are topological invariant for the isolated compact saturated set $S$.  We refer the reader to \cite{Conley}, \cite{Salamon} and \cite{Banyaga-Hurtubise} for the proofs of Propositions \ref{existence_indexpair} and \ref{homotopy_inv_indexpair}

\subsection{The explicit construction of index pairs for the boundary operator}\label{ExplctCnstrctn}

Recall that $M$ is a closed oriented $n$-manifold and that $\varphi_{t}$ is an NMS-flow on $M$ whose associated RHD is simple.  Suppose that $c^{k}$ is a periodic orbit of $\varphi_{t}$ of index $k\ (1\leq k\leq n-1)$ and let $\mathcal{C}(c^{k})$ denote the set of all periodic orbits $c^{k-1}$ of index $k-1$ such that $W(c^{k}, c^{k-1})\neq\emptyset$.  Moreover, let $\mathcal{A}(c^{k})$ be the set of all connected components of $\cup\{ W(c^{k}, c^{k-1})\ |\ c^{k-1}\in\mathcal{C}(c^{k})\}$.  
By the transversality of $W^{u}(c^{k})$ and $W^{s}(c^{k-1})$, $W(c^{k}, c^{k-1})=W^{u}(c^{k})\cap W^{s}(c^{k-1})$ is a $2$-dimensional submanifold.  Note that $W^{u}(c^{k})$ and $W^{s}(c^{k-1})$ are submanifolds and for any point $x\in W(c^{k}, c^{k-1})$ its $\alpha$-limit set $\alpha (x)$ and $\omega$-limit set $\omega (x)$ are equal to $c^{k}$ and $c^{k-1}$ respectively.  Thus every element $A$ of $\mathcal{A}(c^{k})$ is an open saturated annulus and if $A\subset W(c^{k}, c^{k-1})$ then $c^{k}\cup A\cup c^{k-1}$ is a compact saturated set.  We call $A$ a {\em connecting annulus}.  

First, for the clarity, we consider a single connecting annulus $A\in\mathcal{A}(c^{k})$, with $A\subset W(c^{k}, c^{k-1})$.  
Let $R^{k}$ and $R^{k-1}$ be the associated round handles with $c^{k}$ and $c^{k-1}$ respectively and we fix a framing $F: R^{k}\approx S^{1}\times D^{k}\times D^{n-k-1}$.  We define a ``round subhandle'' $Q^{k}\subset R^{k}$ under this framing $F$ as $Q^{k}\approx S^{1}\times D^{k}(r)\times D^{n-k-1}(r)\ (0<r<1)$, where $D^{m}(r)$ denotes the $m$-dimensional disk of radius $r$.  Let $\partial_{-}Q^{k}$ (resp. $\partial_{+}Q^{k}$) denote $S^{1}\times\partial D^{k}(r)\times D^{n-k-1}(r)$\quad (resp. $S^{1}\times D^{k}(r)\times\partial D^{n-k-1}(r)$) under $F$.  Set $a = A\cap\partial_{-}Q^{k}$.  Then, $a\subset W^{u}(c^{k})\cap\partial_{-}Q^{k}\approx S^{1}\times\partial D^{k}(r)\times\{ 0\}$.  
%For the convenience, for a subset $S\subset  M$ let $\mathrm{Sat}^{+}(S)$ (resp. $\mathrm{Sat}^{-}(S)$) denote the non-negative (resp. non-positive) saturation of $S$.  That is, $\mathrm{Sat}^{+}(S) = S\cdot\R_{\geq 0}$ and $\mathrm{Sat}^{-}(S) = S\cdot\R_{\leq 0}$.  We may use the notation $\mathrm{Sat}(S) = S\cdot\R$ as well.  
Then we have a tubular neighborhood $N(a)$ of $a$ in $W^{u}(c^{k})\cap\partial_{-}Q^{k}$ such that $N(a)\subset\partial_{+}R^{k-1}\cdot\R_{\leq 0}$, i.e., $N(a)\cdot\R_{\geq 0}$ pass through $\partial_{+}R^{k-1}$.  Note that in the case $k = 1$, since  $W^{u}(c^{k})\cap\partial_{-}Q^{k} \approx S^{1}\times\{\pm r\}\times\{ 0\}$ and $a$ is a submanifold, $a$ is equal to $S^{1}\times\{ r\}\times\{ 0\}$ or $S^{1}\times\{ -r\}\times\{ 0\}$.  Thus, in this case we set $N(a)=a$.  In any case, $N(a)$ is diffeomorphic to $S^{1}\times D^{k-1}$.  

%%%%%% Claim: isolation of $a$ %%%%%
\begin{claim}\label{a_isolated}
We can choose $N(a)$ so that $N(a)\cap W(c^{k}, c^{k-1}) = a$.  
\end{claim}
\begin{proof}
Recall that $W^{u}(c^{k})$ and $W^{s}(c^{k-1})$ are saturated submanifolds and transverse to each other.  Therefore $W(c^{k}, c^{k-1}) = W^{u}(c^{k})\cap W^{s}(c^{k-1})$ is also saturated submanifold.  Since any flow-line is transverse to $\partial_{-}Q^{k}$ (more precisely, to $\partial_{-}Q^{k}-\partial_{+}Q^{k}$), $W(c^{k}, c^{k-1})\cap\partial_{-}Q^{k}$ is a compact $1$-dimensional submanifold in $\partial_{-}Q^{k}$ and $a$ is its one of connected components.  Thus, if we choose a tubular neighborhood of $a$ small enough, the claim follows. 
\end{proof}
%%%%%%%%% FIg.1, 2 %%%%%%%%%%%%%%%%%%%%%%%%%%
\begin{figure}[htbp]
%\begin{minipage}[b]{6cm}
 	\centering
	\includegraphics[height=7cm]{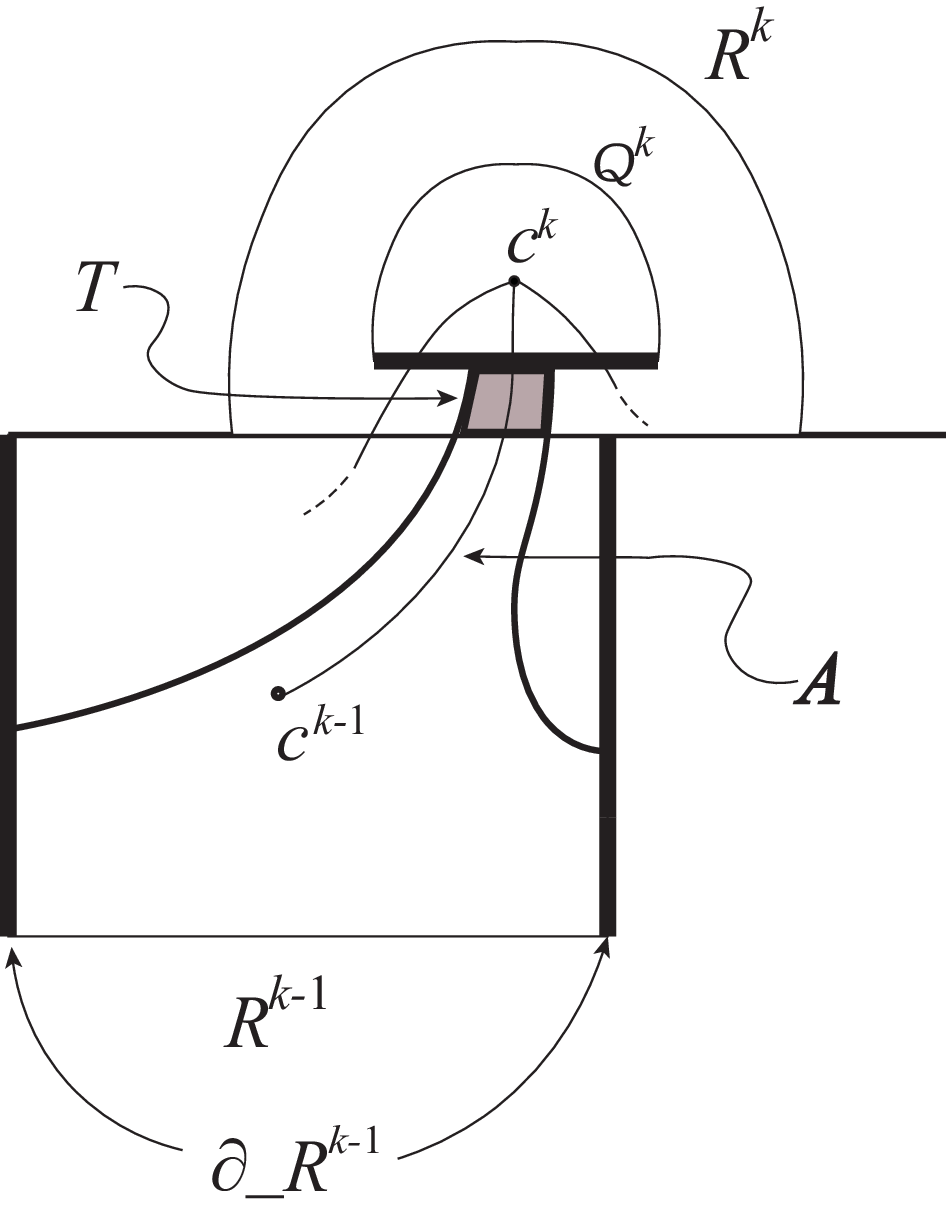}
	\caption{A round subhandle and connecting annuli}
	\label{fig_conn_annulus}
%\end{minipage}
\end{figure}
\begin{figure}
 %\begin{minipage}[b]{6cm}
        \centering 
        \includegraphics[height=8cm]{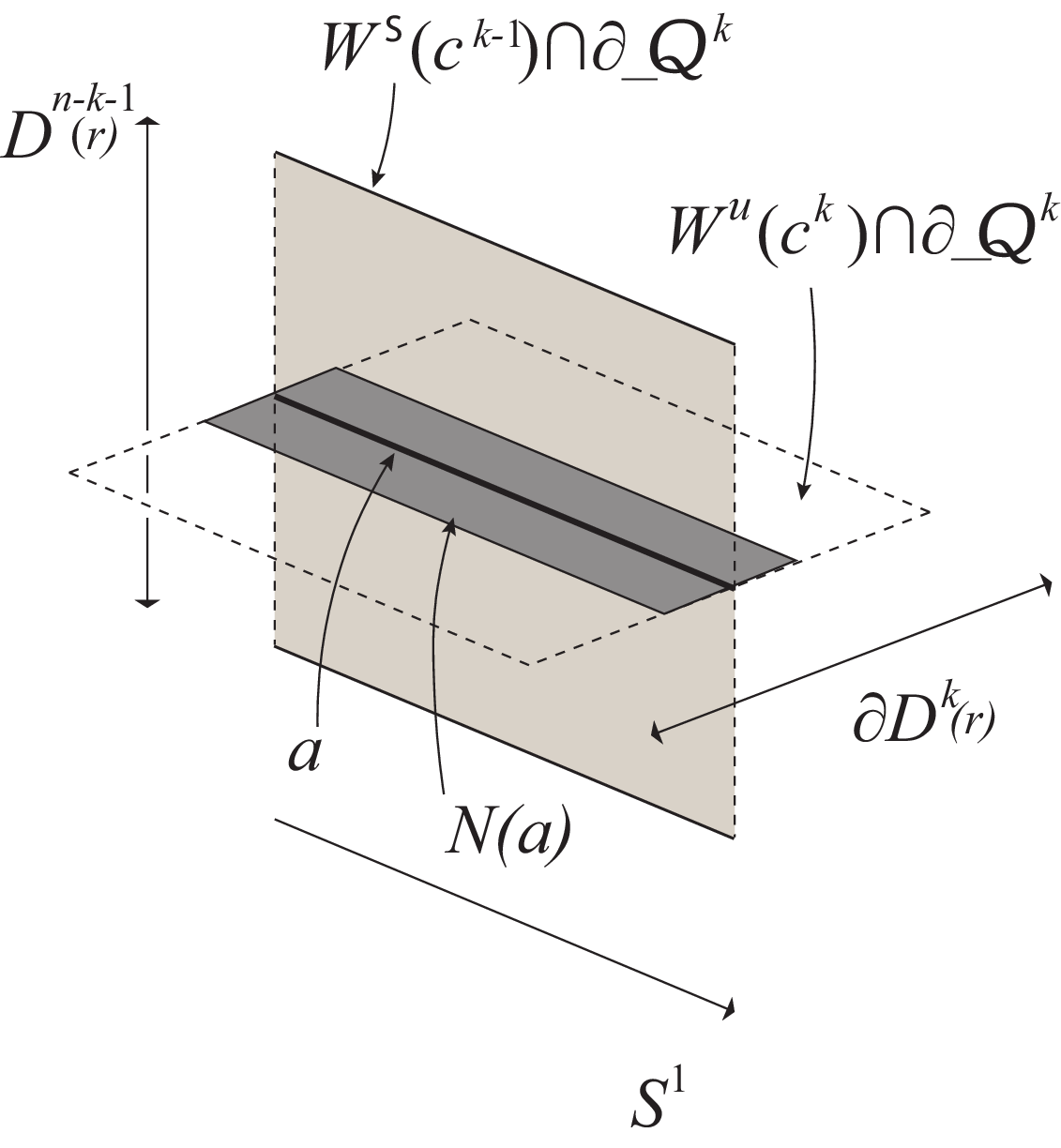}
	\caption{A local view of $\partial_{-}Q^{k}$}
	\label{fig_del_Q^k}
%\end{minipage}
\end{figure}

Recall we have $N(a)\subset W^{u}(c^{k})\cap\partial_{-}Q^{k}$, then by the framing $F$ of $Q^{k}$ we set $\tau = N(a)\times D^{n-k-1}(r)\subset S^{1}\times\partial D^{k}(r)\times D^{n-k-1}(r)$.  Since $N(a)\subset\partial_{+}R^{k-1}\cdot\R_{\leq 0}$, we have $\tau\subset\partial_{+}R^{k-1}\cdot\R_{\leq 0}$ by choosing a smaller $r$ if necessary.  
Then we have $\mathrm{Int}_{\partial_{-}Q^{k}}(\tau) = (\mathrm{Int}_{W^{u}(c^{k})\cap\partial_{-}Q^{k}}(N(a)))\times D^{n-k-1}(r)$, where $\mathrm{Int}_{Y}(B)$ denotes the interior of the subset $B$ in $Y$.  We write $\partial N(a) = N(a)\setminus\mathrm{Int}_{W^{u}(c^{k})\cap\partial_{-}Q^{k}}(N(a))$ and set $\partial_{0}\tau = \partial N(a)\times D^{n-k-1}(r)$.  Note that $\partial_{0}\tau = \tau\setminus\mathrm{Int}_{\partial_{-}Q^{k}}(\tau)$.  Thus $\tau$ is a tubular neighborhood of $a$ in $\partial_{-}Q^{k}$ and $\partial_{0}\tau$ is its ``half boundary'' without $N(a)\times\partial D^{n-k-1}(r)$.  Note that $N(a)\times\partial D^{n-k-1}(r)$ intersects $W^{s}(c^{k-1})$.  Then by the definition of $\tau$ and Claim \ref{a_isolated}, we have the following claim:

%%%%% Claim: $\partial_{0}\tau$ flows through $\partial_{-}R^{k-1}$ %%%%%
\begin{claim}\label{delzerotau_flowthrough}
$\partial_{0}\tau\subset\partial_{-}R^{k-1}\cdot\R_{\leq 0}$, i.e., $\partial_{0}\tau\cap W^{s}(c^{k-1}) = \emptyset$.  
\end{claim}
\noindent
We set $T = (\tau\cdot\R_{\geq 0})\cap (\partial_{+}R^{k-1}\cdot\R_{\leq 0})$.  Note that $T$ is a relative tubular neighborhood of $A-(Q^{k}\cup R^{k-1})$ and diffeomorphic to $S^{1}\times D^{k-1}\times D^{n-k-1}\times [0, 1]$ so that $A\cap T$ equals to $S^{1}\times\{ 0\}\times\{ 0\}\times [0, 1]$ whose last factor $[0, 1]$ corresponds to the flow-line.  

Now we will define compact subsets $N_{0}\subset N_{1}\subset N_{2}$ so that $(N_{2}, N_{1})$ is an index pair for $c^{k}$, $(N_{1}, N_{0})$ is an index pair for $c^{k-1}$ and $(N_{2}, N_{0})$ is an index pair for $c^{k}\cup A\cup c^{k-1}$.  We define 
\[
\begin{array}[t]{ccl}
N_{2} & = & Q^{k}\cup T\cup R^{k-1},\\
N_{1} & = & \partial_{-}Q^{k}\cup T\cup R^{k-1},\\
N_{0} & = & (\partial_{-}Q^{k}\setminus\mathrm{Int}_{\partial_{-}Q^{k}}(\tau))\cup((\partial_{0}\tau\cdot\R_{\geq 0})\cap(T\cup R^{k-1}))\cup\partial_{-}R^{k-1}.
\end{array}
\]
%Here, $\mathrm{Int}_{\partial_{-}Q^{k}}(\tau)$ denotes the interior of $\tau$ in $\partial_{-}Q^{k}$.  
\begin{claim}\label{index_pair_21}
$(N_{2}, N_{1})$ is an index pair for $c^{k}$. 
\end{claim}
\begin{proof}
We have $N_{2}\setminus N_{1} = Q^{k}\setminus\partial_{-}Q^{k}$.  Thus $\mathrm{Cls}(N_{2}\setminus N_{1}) = Q^{k}$ and $\mathrm{Int}(N_{2}\setminus N_{1}) = \mathrm{Int}(Q^{k})$.  Therefore we have $I(\mathrm{Cls}(N_{2}\setminus N_{1})) = I(Q^{k}) = c^{k} \subset\mathrm{Int}(Q^{k}) = \mathrm{Int}(N_{2}\setminus N_{1})$. 

Next, we can rewrite $N_{1}$ as $((\partial_{-}Q^{k}\setminus\tau)\cup \partial_{-}R^{k-1})\sqcup (T\cup (R^{k-1}\setminus\partial_{-}R^{k-1}))$.  If $x\in (\partial_{-}Q^{k}\setminus\tau)\cup\partial_{-}R^{k-1}$, then there is no positive number $t$ such that $x\cdot [0, t]\subset N_{2}$.  If $x\in T\cup (R^{k-1}\setminus\partial_{-}R^{k-1})$, then there is a positive number $t$ such that $x\cdot [0, t]\subset N_{1}$.  

As for (3), if $x\in N_{2}\setminus N_{1} = Q^{k}\setminus\partial_{-}Q^{k}$, then there is a positive number $t$ such that $x\cdot [0, t]\subset Q^{k}\subset N_{2}$.  
\end{proof}
\begin{claim}\label{index_pair_10}
$(N_{1}, N_{0})$ is an index pair for $c^{k-1}$. 
\end{claim}
\begin{proof}
We have $N_{1}\setminus N_{0} = \mathrm{Int}_{\partial_{-}Q^{k}}(\tau)\cup ((T\cup (R^{k-1}\setminus\partial_{-}R^{k-1}))-\partial_{0}\tau\cdot\R_{\geq 0}) = (T\cup (R^{k-1}\setminus\partial_{-}R^{k-1}))-\partial_{0}\tau\cdot\R_{\geq 0}$.  Hence $\mathrm{Cls}(N_{1}\setminus N_{0}) = T\cup R^{k-1}$, $\mathrm{Int}(N_{1}\setminus N_{0}) = \mathrm{Int}(T\cup R^{k-1})-\partial_{0}\tau\cdot\R_{\geq 0}$ and $I(\mathrm{Cls}(N_{1}\setminus N_{0})) = I(T\cup R^{k-1}) = c^{k-1}\subset\mathrm{Int}(T\cup R^{k})-\partial_{0}\tau\cdot\R_{\geq 0} = \mathrm{Int}(N_{1}\setminus N_{0})$.  

If $x\in N_{0}$ and there is a positive number $t$ such that $x\cdot [0, t]\subset N_1$, then $x\in\partial_{0}\tau\cdot\R_{\geq 0}\cap (T\cup (R^{k-1}\setminus\partial_{-}R^{k-1}))$.  Thus $x\cdot [0, t]\subset\partial_{0}\tau\cdot\R_{\geq 0}\cap (T\cup R^{k-1})\subset N_{0}$.  

If $x\in N_{1}\setminus N_{0}$, then there is a positive number $t$ such that $x\cdot [0, t]\subset T\cup R^{k-1}\subset N_{1}$.  
\end{proof}
\begin{claim}\label{index_pair_20}
$(N_{2}, N_{0})$ is an index pair for $c^{k}\cup A\cup c^{k-1}$. 
\end{claim}
\begin{proof}
Since $N_{2}\setminus N_{0} = ((Q^{k}\setminus\partial_{-}Q^{k})\cup T\cup (R^{k-1}\setminus\partial_{-}R^{k-1}))-\partial_{0}\tau\cdot\R_{\geq 0}$, we have $\mathrm{Cls}(N_{2}\setminus N_{0}) = Q^{k}\cup T\cup R^{k-1}$, $\mathrm{Int}(N_{2}\setminus N_{0}) = \mathrm{Int}(Q^{k}\cup T\cup R^{k})-\partial_{0}\tau\cdot\R_{\geq 0}$ and $I(\mathrm{Cls}(N_{2}\setminus N_{0})) = I(Q^{k}\cup T\cup R^{k-1}) = c^{k}\cup A\cup c^{k-1}\subset\mathrm{Int}(Q^{k}\cup T\cup R^{k-1}-\partial_{0}\tau\cdot\R_{\geq 0}) = \mathrm{Int}(N_{2}\setminus N_{0})$.  

If $x\in N_{0}$ and there is a positive number $t$ such that $x\cdot [0, t]\subset N_{2}$, then $x\in (\partial_{0}\tau\cdot\R_{\geq 0})\cap (T\cup (R^{k-1}\setminus\partial_{-}R^{k-1}))$.  Hence $x\cdot [0, t]$ must be contained in $(\partial_{0}\tau\cdot\R_{\geq 0})\cap (T\cup R^{k-1})\subset N_{0}$.  

If $x\in N_{2}\setminus N_{0} = (Q^{k}\setminus\partial_{-}Q^{k})\cup T\cup (R^{k-1}\setminus\partial_{-}R^{k-1})-\partial_{0}\tau\cdot\R_{\geq 0}$, then there is a positive number $t$ such that $x\cdot [0, 1]\subset Q^{k}\cup T\cup R^{k-1} = N_{2}$. 
\end{proof}

\subsection{The explicit construction of index pairs in the general case}
Now, we consider the general situation.  Suppose that the elements of the set $\mathcal{C}(c^{k})$ are indexed by a finite set $I$ and for each $i\in I$ let $\mathcal{A}_{i}(c^{k})$ denote the set of all connecting annuli between $c^{k}$ and $c^{k-1}_{i}$.  Suppose also the elements of $\mathcal{A}_{i}(c^{k})$ are indexed by a finite set $J_{i}$.  Thus, we may write
\[
 \begin{array}[t]{lcl}
  \mathcal{C}(c^{k}) & = & \{ c^{k-1}_{i}\ |\ i\in I\}, \\
  \mathcal{A}(c^{k}) & = & \bigcup_{i\in I}\mathcal{A}_{i}(c^{k}), \\
  \mathcal{A}_{i}(c^{k}) & = & \{ A_{ij}\ |\ j\in J_{i}\},
 \end{array}
\]
where $A_{ij}\subset W(c^{k}, c^{k-1}_{i})\ (j\in J_{i}, i\in I)$ is a connecting annulus.  Let $R^{k-1}_{i}$ be the round $k-1$ handle associated with $c^{k-1}_{i}$.  Then for each $A_{ij}\in\mathcal{A}(c^{k})$ we define the following, similarly as in the single connecting annulus case:
\[
\begin{array}[t]{lcl}
a_{ij} & = & A_{ij}\cap\partial_{-}Q^{k},\\
N(a_{ij}) &  & \mbox{a small tubular neighborhood of $a_{ij}$ in $W^{u}(c^{k})\cap\partial_{-}Q^{k}$}, \\
\tau_{ij} & = & N(a_{ij})\times D^{n-k-1}(r),\\
\partial_{0}\tau_{ij} & = & \partial N(a_{ij})\times D^{n-k-1}(r),\\
T_{ij} & = & (\tau_{ij}\cdot\R_{\geq 0})\cap (\partial_{+}R^{k-1}_{i}\cdot\R_{\leq 0}). 
\end{array}
\]
Now the parallel arguments to the proof of Claim \ref{a_isolated} and \ref{delzerotau_flowthrough} show the following.

\begin{claim}
$N(a_{ij})\cap W(c^{k}, c^{k-1}_{i}) = a_{ij}$. 
\end{claim}
\begin{claim}
$\partial_{0}\tau_{ij}\subset\partial_{-}R^{k-1}_{i}\cdot\R_{\leq 0}$, i.e., $\partial_{0}\tau_{ij}\cap W^{s}(c^{k-1}_{i}) = \emptyset$.  
\end{claim}
\noindent
Also we have
\begin{claim}
$\tau_{ij}\cap\tau_{i'j'} = \emptyset$ if $(i, j) \neq (i', j')$.
\end{claim}
\noindent
Now we use the following notation:
\[
 \begin{array}[t]{lcl}
  \widehat{a} & = & \cup_{i\in I}\cup_{j\in J_{i}}a_{ij},\\
  \widehat{\tau} & = & \cup_{i\in I}\cup_{j\in J_{i}}\tau_{ij},\\
  \partial_{0}\widehat{\tau} & = & \cup_{i\in I}\cup_{j\in J_{i}}\partial_{0}\tau_{ij},\\
  \widehat{R} & = & \cup_{i\in I}R^{k-1}_{i},\\
  \partial_{-}\widehat{R} & = & \cup_{i\in I}\partial_{-}R^{k-1}_{i},\\
  \widehat{A} & = & \cup_{i\in I}\cup_{j\in J_{i}}A_{ij},\\
  \widehat{T} & = & \cup_{i\in I}\cup_{j\in J_{i}}T_{ij}, 
 \end{array}
\]
and 
\[
 \begin{array}[t]{lcl}
  \widehat{N}_{2} & = & Q^{k}\cup\widehat{T}\cup\widehat{R},\\
  \widehat{N}_{1} & = & \partial_{-}Q^{k}\cup\widehat{T}\cup\widehat{R},\\
  \widehat{N}_{0} & = & (\partial_{-}Q^{k}\setminus\mathrm{Int}_{\partial_{-}Q^{k}}\widehat{\tau})\cup((\partial_{0}\widehat{\tau}\cdot\R_{\geq 0})\cap (\widehat{T}\cup\widehat{R}))\cup\partial_{-}\widehat{R}. 
 \end{array}
\]
Then, similar arguments to the proofs of Claims \ref{index_pair_21}, \ref{index_pair_10} and \ref{index_pair_20} show the following.  
\begin{claim}
$(\widehat{N}_{2}, \widehat{N}_{1}), (\widehat{N}_{1}, \widehat{N}_{0})$ and $(\widehat{N}_{2}, \widehat{N}_{0})$ are index pairs for $c^{k}, \cup_{i\in I}c^{k-1}_{i}$ and $c^{k}\cup\widehat{A}\cup (\cup_{i\in I}c^{k-1}_{i})$ respectively.  
\end{claim}

\subsection{A boundary operator along flow-lines}
In this subsection, we calculate the homology groups $H_{m}(\widehat{N}_{2}, \widehat{N}_{1})$ and $H_{m}(\widehat{N}_{1}, \widehat{N}_{0})$ and define a boundary operator obtained from flow-lines themselves.  

\begin{prop}
For any non-negative integer $m$, the relative homology group $H_{m}(\widehat{N}_{2}, \widehat{N}_{1})$ is isomorphic to $H_{m}(Q^{k}, \partial_{-}Q^{k})$.  In particular, $H_{k+1}(\widehat{N}_{2}, \widehat{N}_{1})\cong\Z$.  
\end{prop}
\begin{proof}
For the subset $(\widehat{T}\setminus\widehat{\tau} )\cup \widehat{R}\subset \widehat{N}_{1}$, we have $\mathrm{Cls}_{\widehat{N}_{2}}((\widehat{T}\setminus\widehat{\tau} )\cup \widehat{R}) = \widehat{T}\cup \widehat{R}$ and $\mathrm{Int}_{\widehat{N}_{2}}(\widehat{N}_{1}) = \widehat{T}\cup \widehat{R}$.  Thus $\mathrm{Int}_{\widehat{N}_{2}}(\widehat{N}_{1})\subset\mathrm{Cls}_{\widehat{N}_{2}}((\widehat{T}\setminus\widehat{\tau} )\cup \widehat{R})$, and by the excision we have 
\begin{eqnarray*}
H_{m}(\widehat{N}_{2}, \widehat{N}_{1}) & \cong & H_{m}(\widehat{N}_{2}\setminus ((\widehat{T}\setminus\widehat{\tau})\cup \widehat{R}), \widehat{N}_{1}\setminus ((\widehat{T}\setminus\widehat{\tau})\cup \widehat{R})) \\
 & = & H_{m}(Q^{k}, \partial_{-}Q^{k}).  
\end{eqnarray*}
\end{proof}

\begin{prop}
For any non-negative integer $m$, the relative homology group $H_{m}(\widehat{N}_{1}, \widehat{N}_{0})$ is isomorphic to $H_{m}(\widehat{R}, \partial_{-}\widehat{R})$.  In particular, $H_{k}(\widehat{N}_{1}, \widehat{N}_{0})$ is the free Abelian group of rank $|I|$, where $|I|$ denotes the number of elements of the index set $I$ of $\mathcal{C}(c^{k})$.  
\end{prop}
\begin{proof}
Let $\widetilde{\tau}$ be an expanded tubular neighborhood of $\widehat{a}$ in $\partial_{-}Q^{k}$ such that $\widehat{a}\subset\widehat{\tau}\subset\mathrm{Int}_{\partial_{-}Q^{k}}(\widetilde{\tau})$.  Then for the subset $\partial_{-}Q^{k}\setminus\widetilde{\tau}\subset \widehat{N}_{0}$, we have $\mathrm{Cls}(\partial_{-}Q^{k}\setminus\widetilde{\tau}) = \partial_{-}Q^{k}\setminus\mathrm{Int}_{\partial_{-}Q^{k}}(\widetilde{\tau})$ and $\mathrm{Int}_{\widehat{N}_{1}}(\widehat{N}_{0}) = \partial_{-}Q^{k}\setminus\widehat{\tau}$.  Thus $\mathrm{Cls}(\partial_{-}Q^{k}\setminus\widetilde{\tau})\subset\mathrm{Int}_{\widehat{N}_{1}}(\widehat{N}_{0})$, therefore by the excision and homotopy equivalences,  we have
\begin{eqnarray*}
H_{m}(\widehat{N}_{1}, \widehat{N}_{0}) & \cong & H_{m}(\widehat{N}_{1}\setminus (\partial_{-}Q^{k}\setminus\widetilde{\tau}), \widehat{N}_{0}\setminus (\partial_{-}Q^{k}\setminus\widetilde{\tau})) \\
 & = & H_{m}(\widetilde{\tau}\cup \widehat{T}\cup \widehat{R}, (\widetilde{\tau}\setminus\mathrm{Int}_{\partial_{-}Q^{k}}(\widehat{\tau}))\cup ((\partial_{0}\widehat{\tau}\cdot\R_{\geq 0})\cap (\widehat{T}\cup \widehat{R}))\cup\partial_{-}\widehat{R}) \\
 & \cong & H_{m}(\widehat{R}, (\partial_{0}\widehat{\tau}\cdot\R_{\geq 0}\cap \widehat{R})\cup\partial_{-}\widehat{R}) \\
 & \cong & H_{m}(\widehat{R}, \partial_{-}\widehat{R}).
\end{eqnarray*}
Indeed, the isomorphism in the first line is the excision, the equality in the second line is just the notation rewritten, the third one is by contracting $\widetilde{\tau}\cup \widehat{T}$, and the last one is induced by the homotopy equivalence defined as follows: 
For simplicity of the notation, set $Y=\widehat{R}$, $B=(\partial_{0}\widehat{\tau}\cdot\R_{\geq 0}\cap \widehat{R})\cup\partial_{-}\widehat{R}$ and $C=\partial_{-}\widehat{R}$.  Let $h_{t}:B\rightarrow Y$ be the homotopy defined by shrinking $(\partial_{0}\widehat{\tau}\cdot\R_{\geq 0})\cap \widehat{R}$ along the flow-lines starting at $\partial_{0}\widehat{\tau}$ so that $h_{0}=\mathrm{id}_{Y}|B$, $h_{1}(B)=C$ and $h_{t}|C$ is the inclusion $C\hookrightarrow Y$.  Then since $B\hookrightarrow Y$ is a cofibration, we have an extension $\hat{h}_{t}:Y\rightarrow Y$ of $h_{t}$.  Hence we have a map $\hat{h}_{1}:Y\rightarrow Y$ which is homotopic to the identity, $\hat{h}_{1}(B)=C$ and $\hat{h}_{1}|C$ is the inclusion $C\hookrightarrow Y$.  Then $\hat{h}_{1}:(Y, B)\rightarrow (Y, C)$ is a homotopy equivalence with homotopy inverse $\mathrm{id}_{Y}:(Y, C)\rightarrow (Y, B)$.  Thus $\hat{h}_{1}$ induces the last isomorphism.  
\end{proof}

In view of these propositions, we can regard $H_{k+1}(\widehat{N}_{2}, \widehat{N}_{1})$ and $H_{k}(\widehat{N}_{1}, \widehat{N}_{0})$ as free Abelian groups generated by periodic orbits $c^{k}$ and $\{ c^{k-1}_{i}\ |\ i\in I\}$ respectively.  Note that since $(\widehat{N}_{2}, \widehat{N}_{1})$ and $(\widehat{N}_{1}, \widehat{N}_{0})$ are (regular) index pairs for $c^{k}$ and $\cup_{i\in I}c^{k-1}_{i}$ respectively, $H_{k+1}(\widehat{N}_{2}, \widehat{N}_{1})$ and $H_{k}(\widehat{N}_{1}, \widehat{N}_{0})$ are invariants of $c^{k}$ and $\cup_{i\in I}c^{k-1}_{i}$ respectively.  
Next let $\emptyset = M^{(-1)}\subset M^{(0)}\subset M^{(1)}\subset\cdots\subset M^{(n-1)} = M$ be the filtration obtained from the NMS flow as in the Section 3.  
Then we define continuous maps $f: (\widehat{N}_{2}, \widehat{N}_{1})\rightarrow (M^{(k)}, M^{(k-1)})$ and $g: (\widehat{N}_{1}, \widehat{N}_{0})\rightarrow (M^{(k-1)}, M^{(k-2)})$ as follows.  
Let $C(\partial_{-}Q^{k})$ denote a (relative) open collar of $\partial_{-}Q^{k}$ in $Q^{k}$, i.e., $S^{1}\times\partial_{-}D^{k}\times(D^{n-k-1}(r)\setminus D^{n-k-1}(r-\epsilon))$ with respect to the framing $F$.  Let $f|Q^{k}\setminus C(\partial_{-}Q^{k})$ and $f|\widehat{R}$ be the inclusions $Q^{k}\setminus C(\partial_{-}Q^{k})\hookrightarrow M^{(k)}$ and $\widehat{R}\hookrightarrow M^{(k-1)}$ respectively.  $f|C(\partial_{-}Q^{k})\cup\widehat{T}$ is defined to be the map which is pushing $\partial_{-}Q^{k}$ and shrinking $\widehat{T}$ into $\widehat{R}$ along flow-lines passing through $\partial_{-}Q^{k}$.  Thus $f(\widehat{N}_{1})\subset\widehat{R}\subset M^{(k-1)}$ and $f$ is homotopic to the inclusion $\widehat{N}_{2}\hookrightarrow M^{(k)}$.  As for the map $g$, let $h:\widehat{R}\rightarrow\widehat{R}$ be a continuous map pushing $((\partial_{-}Q^{k}\setminus\mathrm{Int}_{\partial_{-}Q^{k}}\widehat{\tau})\cdot\R_{\geq 0})\cap\widehat{R}$ into $\partial_{-}\widehat{R}$ along flow-lines similarly as in the definition of $f$.  We define $g$ to be the composition $h\circ(f|\widehat{N}_{1})$.  Then, $g(\widehat{N}_{0})\subset\partial_{-}\widehat{R}\subset M^{(k-2)}$ and $g$ is homotopic to $f|\widehat{N}_{1}$.  It is easy to see the following diagram is commutative:
\begin{equation}\label{big_diagram}
 \vcenter{
 \xymatrix@C=0pt{
  H_{k+1}(\widehat{N}_{2}, \widehat{N}_{1}) \ar[rrrr]^{f_{*}}  \ar[rdd]^{\partial} \ar[dddd]_{\partial}& & & & H_{k}(M^{(k)}, M^{(k-1)}) \ar[ldd]_{\partial} \ar[dddd]^{\partial} \\
     & & H_{k+1}(Q^{k}, \partial_{-}Q^{k}) \ar[llu]^{\cong}_{\mathrm{exc}} \ar[rru] \ar[dd]^{\partial} & & \\
     & H_{k}(\widehat{N}_{1}) \ar[ldd] & & H_{k}(M^{(k-1)}) \ar[rdd] & \\
     & & H_{k}(\partial_{-}Q^{k})\ar[lu] \ar[ru]_{(f|\partial_{-}Q^{k})_{*}} & &\\
  H_{k}(\widehat{N}_{1}, \widehat{N}_{0}) \ar[rrrr]^{g_{*}} & & & & H_{k}(M^{(k-1)}, M^{(k-2)})  
  }
}
\end{equation}
Here, the bottom pentagon is commutative since $g|\partial_{-}Q^{k}$ is homotopic to $f|\partial_{-}Q^{k}$.  Note that the homomorphisms $f_{*}$ and $g_{*}$ are monomorphisms.  Now we write the generators of $H_{k+1}(\widehat{N}_{2}, \widehat{N}_{1})$ and $H_{k}(\widehat{N}_{1}, \widehat{N}_{0})$ as $\langle c^{k}\rangle$ and $\langle c^{k-1}_{i}\rangle$ etc. and suppose that the image of $\langle c^{k}\rangle$ by the connecting homomorphism $\partial_{k+1}: H_{k+1}(\widehat{N}_{2}, \widehat{N}_{1})\rightarrow H_{k}(\widehat{N}_{1}, \widehat{N}_{0})$ is expressed as a linear combination $\partial_{k+1}(\langle c^{k}\rangle) = \sum_{i\in I}[c^{k};c^{k-1}_{i}]\langle c^{k-1}_{i}\rangle\quad ([c^{k};c^{k-1}_{i}]\in\Z )$.  On the basis of this expression, we introduce the following:  Let $\mathcal{F}$ denote the NMS-foliation defined by the flow $\varphi_{t}$.  Recall that $\mathcal{F}$ is assumed to be simple which means the associated RHD is simple.  
Formally, we define $C^{^\mathrm{NMS}}_{k}(\mathcal{F})$ to be a free Abelian group generated by all compact leaves of $\mathcal{F}$(i.e., all periodic orbits of $\varphi_{t}$) of index $k$.  Moreover, we define
\[
 \partial^{^\mathrm{NMS}}_{k}(\langle c^{k}\rangle ) = \sum_{c^{k-1}\in\mathcal{C}(c^{k})}[c^{k};c^{k-1}]\langle c^{k-1}\rangle
\]
so that we have a homomorphism $\partial^{^\mathrm{NMS}}_{k}: C^{^\mathrm{NMS}}_{k}(\mathcal{F})\rightarrow C^{^\mathrm{NMS}}_{k-1}(\mathcal{F})$.  
By the definition of $\partial^{\mathrm{NMS}}_{k}$ and the commutativity of the outermost square of the diagram (\ref{big_diagram}), we have the following commutative diagram:
\begin{equation}
\vcenter{
\xymatrix@C=60pt{
 C^{^\mathrm{NMS}}_{k}(\mathcal{F}) \ar[rr]^{\cong} \ar[ddd]_{\partial^{^\mathrm{NMS}}_{k}}& & H_{k+1}(M^{(k)}, M^{(k-1)}) \ar[ddd]^{\partial_{k+1}}\\
  & H_{k+1}(\widehat{N}_{2}, \widehat{N}_{1}) \ar[d]^{\partial_{k+1}} \ar[lu] \ar[ru]^{f_{*}} & \\
  & H_{k}(\widehat{N}_{1}, \widehat{N}_{0}) \ar[ld] \ar[rd]^{g_{*}} & \\
 C^{^\mathrm{NMS}}_{k-1}(\mathcal{F}) \ar[rr]^{\cong} & & H_{k}(M^{(k-1)}, M^{(k-2)})
}
} 
\end{equation}
Hence we have $\partial^{^\mathrm{NMS}}_{k-1}\circ\partial^{^\mathrm{NMS}}_{k}=0$ so that $\{ C^{^\mathrm{NMS}}_{*}(\mathcal{F}), \partial^{^\mathrm{NMS}}_{*}\}$ is a chain complex.  Let $H^{\mathrm{NMS}}_{*}(\mathcal{F})$ denote its homology and we call it the {\em homology} of NMS-foliation $\mathcal{F}$ on $M$.  The homology of an NMS-foliation is isomorphic to the homology of its associated RHD since they are defined as isomorphic chain complexes.  Thus we have

\begin{Thm}
For a simple NMS-foliation $\mathcal{F}$ on $M$, a homology $H^{\mathrm{NMS}}_{*}(\mathcal{F})$ of $\mathcal{F}$ is defined which is isomorphic to the homology of the RHD associated with $\mathcal{F}$. 
\end{Thm}

\section{Calculations for NMS-foliations associated with Seifert fibrations}\label{SF}

In this section, we calculate the homology groups of NMS-foliations which are naturally associated with Seifert fibrations. Recall that we assume everything is orientable, for simplicity.  Let $M$ be a closed connected 3-manifold and $p:M\rightarrow\Sigma$ a Seifert fibration over closed surface $\Sigma$. Seifert fibrations are classified by the {\em Seifert invariants} $(g; \beta_{1}/\alpha_{1}, \beta_{2}/\alpha_{2}, \ldots , \beta_{m}/\alpha_{m})$.  Here, $g$ denotes the genus of $\Sigma$ and $(\alpha_{i}, \beta_{i})$ is a {\em Seifert pair} of a fiber.  More precisely, let $D_{1},\ldots , D_{m}$ be a non-empty collection of disjoint closed disks in $\Sigma$ so that $\check{M} = p^{-1}(\Sigma\setminus\cup_{i=1}^{m}\mathrm{Int}D_{i})\stackrel{p}{\rightarrow}\check{\Sigma} = \Sigma\setminus\cup_{i=1}^{m}\mathrm{Int}D_{i}$ is a locally trivial fibration admitting a section $s:\check{\Sigma}\rightarrow \check{M}$, while each $p^{-1}(D_{i})\approx D^{2}\times S^{1}$ is a solid torus, called a {\em fibered solid torus}.  The Seifert pair $(\alpha_{i}, \beta_{i})$ which is associated with each $D_{i}$ is a coprime integers, with $\alpha_{i}\geq 1$ so that $[s(\partial D_{i})] = \beta_{i}\in\Z\cong\pi_{1}(p^{-1}(D_{i}))$ and fibering $p$ over $D_{i}$ is given by 
\[
p^{-1}(D_{i})\approx D^{2}\times S^{1}\rightarrow D^{2}\ ;\ (re^{\sqrt{-1}\theta}, e^{\sqrt{-1}\phi})\mapsto re^{\sqrt{-1}(\alpha_{i}\theta -\nu_{i}\phi)}
\]
where $\nu_{i}$ is an inverse of $\beta_{i}$ modulo $\alpha_{i}$ and $D^{2}$ and $D_{i}$ are considered as the unit disk in $\C$.  In this fibered solid torus, we call the core circle $p^{-1}(0)$\smallskip\ a {\em regular fiber} in the case $\alpha_{i} = 1$ and an {\em exceptional fiber} in the case $\alpha_{i} > 1$.  Thus we have 
$\left[\begin{array}{cc}
       \ell      & \nu_{i} \\
       \beta_{i} & \alpha_{i}	
 \end{array}\right]$\smallskip
as an attaching diffeomorphism (or framing-change) from $\partial_{i}\check{M} = \partial p^{-1}(D_{i})$ with the section-regular-fiber framing to $\partial (p^{-1}(D_{i}))\approx\partial D^{2}\times S^{1}$ with the meridian-longitude ($[\partial D^{2}\times\{ 1\}]-[\{ 1\}\times S^{1}]$) framing.  Here, $\ell$ is an integer satisfying $\nu_{i}\beta_{i} - 1 = \ell\alpha_{i}$ and note that the diffeomorphism is orientation-reversing.  Then the Seifert invariant classifies Seifert fibrations as follows (see \cite{Neumann-Raymond} and \cite{Eisenbud-Hirsch-Neumann}).  

\begin{prop}
Let $p:M\rightarrow\Sigma$ and $p':M'\rightarrow\Sigma'$ be two Seifert fibrations with associated Seifert invariants $(g;\beta_{1}/\alpha_{1},\ldots , \beta_{m}/\alpha_{m})$ and $(g';\beta'_{1}/\alpha'_{1},\ldots , \beta'_{m'}/\alpha'_{m'})$ respectively.  Then $p$ and $p'$ are orientation preservingly diffeomorphic by a fiber preserving diffeomorphism if and only if, after reindexing if necessary, there exists $n$ such that
\begin{enumerate}
\item 
$\alpha_{i} = \alpha'_{i}$ for $i = 1,\ldots , n$ and $\alpha_{j} = \alpha'_{j'} = 1$ for $j, j' > n$,
\item
$\beta_{i}\equiv\beta'_{i}\mod\alpha_{i}$ for $i = 1, \ldots , n$,
\item
${\displaystyle \sum_{i=1}^{m}\beta_{i}/\alpha_{i} = \sum_{i=1}^{m'}\beta'_{i}/\alpha'_{i}}$.  
\end{enumerate}
\end{prop}

In order to associate naturally an NMS-foliation with a Seifert fibration $p:M\rightarrow\Sigma$, we use the notion of {\em round Morse functions}, which are also called Morse-Bott functions whose critical submanifolds are one-dimensional.  That is, a round Morse function is a function whose critical set is a disjoint union of embedded circles and at a critical point it is a Morse type singular point in the transverse direction to the critical circle. For the details, we refer the readers to \cite{mytop} and  \cite{Banyaga-Hurtubise} for example.  Suppose that $p:M\rightarrow\Sigma$ is a Seifert fibration and $f:\Sigma\rightarrow\R$ is a Morse function.  We assume that each exceptional fiber of $p$ lies over a critical point of index $0$ of $f$.  Then the composition $f\circ p:M\rightarrow\R$ is a round Morse function.  Note that exceptional fibers with Seifert pair $(2, \beta)$ are allowed to be over critical points in the case twisted round handles are admitted.  Choosing a Riemann metric $\langle - , -\rangle$ on $M$, we have a gradient vector field $\mathrm{grad}(f\circ p)$ by $\langle\mathrm{grad}(f\circ p), -\rangle = -d(f\circ p)$ and we obtain an NMS flow from $\mathrm{grad}(f\circ p)$ by pushing it along all critical circles.  That is, let $\theta$ denote the coordinate of the second factor of the framing $D^{2}\times S^{1}\approx p^{-1}(D_{i})$ of a fibered solid torus and $X$ a vector field on $M$ which is equal to $\frac{\partial}{\partial\theta}$ near the critical circles and zero away from them.  Then, it is easily seen that $\mathrm{grad}(f\circ p) + X$ is an NMS flow on $M$.  Note that the homology of the NMS-foliation defined by this NMS flow is determined essentially by the Morse function $f:\Sigma\rightarrow\R$.  

Now suppose that $(g;\beta_{1}/\alpha_{1},\ldots , \beta_{m}/\alpha_{m})$ is the Seifert invariant of a Seifert fibration $p:M\rightarrow\Sigma$ each exceptional fiber of which lies over a critical point of index $0$ of a Morse function $f:\Sigma\rightarrow\R$.  Moreover, $f$ has only one critical point of index $2$.  
By our setting, we may see the RHD of $M$ through the (ordinary) handle decomposition of $\Sigma$ with $f$.  Thus, if $\Sigma = (\sum_{i=1}^{m}h^{0}_{i})+(\sum_{i=1}^{m-1}h^{1}_{i})+(\sum_{i=1}^{2g}h^{1}_{m+i-1})+h^{2}$ is the handle decomposition of $\Sigma$ with $f$, for example, the following figure may be considered as a description of the RHD of $M$ which is the pull-back of this handle decomposition by $p$.  Note that $h^{k}_{i}$ denotes an ordinary $k$-handle and in fact the figure shows $\Sigma\setminus\mathrm{Int}(h^{2})$.  Let $\mathcal{F}(p, f)$ denote the oriented one-dimensional foliation on $M$ obtained from the associated NMS flow as above.  
\begin{figure}[ht]
	\centering
	\includegraphics[height=5cm]{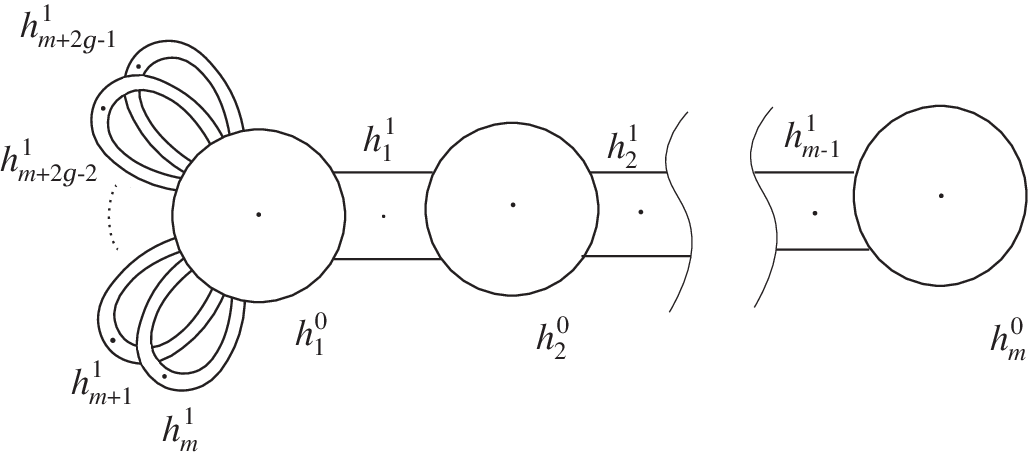}
	\caption{A handle decomposition of $\Sigma\setminus\mathrm{Int}(h^{2})$ for a Seifert fibration $p:M\rightarrow\Sigma$}
	\label{fig_SF_RHD/HD}
\end{figure}
In this example, our chain complex 
$0\rightarrow C^{^\mathrm{RHD}}_{2}\stackrel{\hspace{1em}\partial^{^\mathrm{RHD}}_{2}}{\longrightarrow}C^{^\mathrm{RHD}}_{1}\stackrel{\hspace{1em}\partial^{^\mathrm{RHD}}_{1}}{\longrightarrow}C^{^\mathrm{RHD}}_{0}\rightarrow 0$ 
can be calculated as follows:
The chain groups are
\[
\left\{\begin{array}[]{lcl}
C^{^\mathrm{RHD}}_{2} & = & \Z \medskip\\
C^{^\mathrm{RHD}}_{1} & = & \Z^{m+2g-1} \medskip\\
C^{^\mathrm{RHD}}_{0} & = & \Z^{m}  
\end{array}\right.
\]
where $\Z^{k}$ denotes the direct sum of $k$ infinite cyclic groups, i.e. the free Abelian group of rank $k$.  The boundary operators are $\partial^{^\mathrm{RHD}}_{2}=0$ and $\partial^{^\mathrm{RHD}}_{1}=0$ on the last $2g$ summands and it is expressed on the first $m-1$ summands as 
\[
\partial^{^\mathrm{RHD}}_{1}|\Z^{m-1} = 
\left[\begin{array}{ccccc}
 \alpha_{1} &  0          & \cdots & \cdots &  0 \\
-\alpha_{2} &  \alpha_{2} & \ddots &        &  \vdots \\
 0          & -\alpha_{3} & \ddots & \ddots &  \vdots \\
\vdots      & \ddots      & \ddots & \ddots &  0 \\
\vdots      &             & \ddots & \ddots &  \alpha_{m-1} \\         
0           & \cdots      & \cdots & 0      & -\alpha_{m}
\end{array}
\right]
\]
Suppose that the elementary divisors of $\partial^{^\mathrm{RHD}}_{1}$ are $e_{1},\ldots, e_{r}$. Since $\alpha_{i}\geq 1$, we have $r = m-1$ and the following:\[
 \left\{\begin{array}{lcl}
  H^{^\mathrm{NMS}}_{2}(\mathcal{F}(p, f)) & \cong & \Z \\
  H^{^\mathrm{NMS}}_{1}(\mathcal{F}(p, f)) & \cong & \Z^{2g} \\
  H^{^\mathrm{NMS}}_{0}(\mathcal{F}(p, f)) & \cong & (\Z/e_{1}\Z)\oplus(\Z/e_{2}\Z)\oplus\cdots\oplus(\Z/e_{m-1}\Z)\oplus\Z
        \end{array}\right. 
\]
It is easily seen that if $\alpha_{1},\ldots,\alpha_{m}$ are mutually coprime, then the elementary divisors are all equal to $1$, hence $H^{^\mathrm{NMS}}_{0}(\mathcal{F}(p, f))\cong\Z$, and $H^{^\mathrm{NMS}}_{*}(\mathcal{F}(p, f))\cong H_{*}(\Sigma)$.

%%%%%%%%%%%%%%%%%%%%%%%%%%%%%%%%%%%%%%%
%\begin{ack}[Acknowledgement]
%The authors would like to thank the anonymous referees for their careful readings and suggestions of corrections of the errors in the manuscript. 
%\end{ack}
%
%%%%%%%%%%%%%%%%%%%%%%%%%%%%%%%%%%%%%%%

%\clearpage

\end{document}